\documentclass[reqno, 12pt]{amsart}

\pdfoutput=1

\usepackage{enumerate}
\usepackage{latexsym}
\usepackage[centertags]{amsmath}
\usepackage{amsfonts}
\usepackage{amssymb}
\usepackage{amsthm}
\usepackage{newlfont}
\usepackage{graphics}
\usepackage{color}
\usepackage{float}
\usepackage{diagbox}
\usepackage{longtable}
\usepackage{rotating}
\usepackage{multirow}

\usepackage{extarrows}

\usepackage[sort,compress,numbers]{natbib}

\usepackage[utf8]{inputenc}

\newtheorem{theo}{Theorem}

\newtheorem{example} [theo]{Example}
\newtheorem{lem} [theo]{Lemma}
\newtheorem{cor}[theo]{Corollary}
\newtheorem{prop}[theo]{Proposition}
\numberwithin{equation}{section}

\def\Z{\mathbb{Z}}

\newcommand{\ZZ}{{\mathbb{Z}}}

\newcommand{\U}{{\textsf{U}}}
\newcommand{\DD}{{\textsf{D}}}

\renewcommand{\P}{{\mathbb{P}}}

\title[Hankel Determinants for Dyck Paths]{Proof of a Conjecture on Hankel Determinants \\for Dyck Paths with Restricted Peak Heights}

\author{ Guoce Xin$^{1,*}$ and Zihao Zhang$^{2}$}

\address{ $^{1,2}$School of Mathematical Sciences, Capital Normal University,
 Beijing 100048, PR China}
\email{$^1$\texttt{guoce.xin@163.com }\ \  \& $^2$\texttt{zihao-zhang@foxmail.com}}
\date{ \today}
\thanks{$*$ This work was partially supported by NSFC(12071311).}
\begin{document}
\maketitle

\begin{abstract}
For any integer $m\geq 2$ and $r \in \{1,\dots, m\}$, let $f_n^{m,r}$ denote the number of $n$-Dyck paths whose peak's heights are $im+r$ for
some integer $i$. We find the generating function of $f_n^{m,r}$ satisfies a simple algebraic functional equation of degree $2$. The
$r=m$ case is particularly nice and we give a combinatorial proof.
By using the Sulanke and Xin's continued fraction method, we calculate the Hankel determinants for $f_n^{m,r}$. The special case $r=m$
of our result solves a conjecture proposed by Chien, Eu and Fu. We also enriched the class of eventually periodic Hankel determinant sequences.
\end{abstract}

\def\D{{\mathcal{D}}}

\noindent
\begin{small}
 \emph{Mathematic subject classification}: Primary 05A15; Secondary 05A15, 11B83.
\end{small}

\noindent
\begin{small}
\end{small}

\section{Introduction}
 \subsection{Dyck Paths with Restricted Peak Heights}

    An \emph{$n$-Dyck path} is a lattice path from $(0,0)$ to $(2n,0)$, with unit steps either an up step $\U=(1,1)$ or a down step $\DD=(1,-1)$, staying weakly above the $x$-axis. The number of $n$-Dyck paths is counted by the celebrated $n$th Catalan number $C_n=\frac{1}{n+1}\binom{2n}{n}$, which has more than 200 combinatorial interpretations \cite{EC2}. Its generating function $c(x):=\sum_{n\geq 0}{C_n x^n}$ is uniquely determined by the quadratic functional equation $ c(x)=1+xc(x)^2$.

In a Dyck path $D$, an up-step followed by a down-step is called a \emph{peak}; a down-step followed by an up-step is called a \emph{valley}; the \emph{height} of a peak (or valley) is the $y$-coordinate of the intersection point of its two steps. Given a set $S\subset \ZZ$ of heights, denote by $\D^S=\cup_n \D_n^S$ with $\D_n^S$ being the set of all $n$-Dyck paths whose peak heights are not in $S$. A Dyck path $D\in \D^s$ is also called with peaks avoiding heights in $S$.
Clearly $\D^S=\D^{S\cap \P}$ where and in what follows, $\P$ always denotes the set of positive integers.

Denote by
$d_n^S=|\D_n^S| $ the cardinality of $\D_n^S$. The generating function of $d_n^S$ is
$$D^S = D^S(x) = \sum_{n\ge 0} d_n^S x^n.$$
For example, if $S=\emptyset$, then elements in $\D_n^S$ are just ordinary $n$-Dyck paths counted by the Catalan numbers $C_n$.

For given $S\subset \ZZ$ and $k\in \ZZ$, we use the notation $S-k=\{s-k: s\in S\}$.
The ``first return decomposition" applies to nonempty Dyck paths $\pi\in \D^S$: $\pi=\U\mu\DD\nu$, i.e., $\pi$ is decomposed as an up-step $\U$ followed by a Dyck path $\mu$, followed by a down step $\DD$, followed by a Dyck path $\nu$. It is clear that $\mu\in \D^{S-1}$ and $\nu\in \D^{S}$; moreover,
 $\mu$ has to be non-empty if $1\in S$. In terms of generating functions, we have
 $$ D^S=1+x (D^{S-1}-\chi(1\in S))  D^S,$$
 where we have used the notation $\chi(\text{true})=1$ and $\chi(\text{false})=0$. Equivalently, we have
 \begin{equation}\label{D22D}
D^{S}=
           \dfrac{1}{1+\chi(1\in S)x-xD^{S-1}}.
\end{equation}

The case $S=S+m$ for some positive integer $m$ is particularly interesting. The smallest such $m$ is called the period of $S$ and in this case
$S=V+m\ZZ$ for some set $V$, which can be chosen to be a proper subset of $[m]:=\{1,2,\dots, m\}$. This model has been studied by
Chien et al. \cite{CPF}. They used $(m,V)$ to denote $V+m\ZZ$. Then by iteratively applying \eqref{D22D}, one obtains
\begin{equation} \label{basiccf}
D^{(m,V)}=\dfrac{1}{1+\chi(1\in V)x-\dfrac{x}{\qquad\dfrac{\ddots}{1+\chi(m-1\in V)x-\dfrac{x}{1+\chi(m\in V)x-xD^{(m,V)}}}}}.
\end{equation}
It is easy to see that $D^{(m,V)}$ is algebraic of degree $2$, which is a crucial fact in our evaluation of the Hankel determinants.
For instance,
i) if $S=2\ZZ+1$ is the set of odd integers, then $S-2=S$ and
we have
$$ D^S = \dfrac{1}{1+x-\dfrac{x}{1-x D^S}} \Rightarrow  D^{2\ZZ+1}= {\frac {x+1-\sqrt {-3\,{x}^{2}-2\,x+1}}{2x \left( 1+x \right) }};$$
ii) if $S=2\ZZ$ is the set of even integers, then we have
$$ D^S = \dfrac{1}{1-\dfrac{x}{1+x-x D^s}} \Rightarrow  D^{2\ZZ}= {\frac {x+1-\sqrt {-3\,{x}^{2}-2\,x+1}}{2x}}.$$
It is known \cite{ELY} that: $\D_n^{2\ZZ+1}$ and $\D_n^{2\ZZ}$ are counted by Riordan numbers and shifted Motzkin numbers, respectively. They contains all $n$-Dyck paths with all peaks at even height and odd height, respectively.

\subsection{Hankel Determinants for Generating Functions}
Given a generating function $A(x)=\sum_{n\geq 0}a_nx^n$, its $k$-shifted Hankel determinants is defined by
$$ H_n^{k}(A(x)) = \det ( a_{i+j+k} )_{0\le i,j\le n-1},\qquad H_0^k(A(x))=1.$$
Hankel determinants evaluation has a long history. It usually refers to finding nice formula of $H_n(A(x)):=H_n^0(A(x))$ for general $n\ge 0$.
Many methods have been developed, such as Gauss's continued fraction, the method of orthogonal polynomials. See \cite{gauss,GX,Cigler,C. Krattenthaler.}.

This paper is along the line of using generating functions to deal with Hankel determinants.
Classical method of continued fractions, either by $J$-fractions (Krattenthaler \cite{C. Krattenthaler.} or Wall \cite{H. S. Wall}),
or by $S$-fractions (Jones and Thron \cite[Theorem 7.2]{W. B. Jones and W. J. Thron}), requires $H_n(A(x))\neq 0$ for all $n$.
Gessel-Xin's \cite{GX} continued fraction method allows $H_n(A(x))=0$ for some values of $n$. Their method is based on three rules
that act on two variable generating functions
and transform one set of determinants to another set of determinants of the same values.
See Section \ref{GXCF}. We remark that Han's continued fractions also allows $H_n(A(x))=0$. See \cite{Han}.

Recently in \cite{CPF}, Chien-Eu-Fu studied the generating function $D^{m,V}$. They found plenty of cases such that the sequence of Hankel determinants $H_n(D^{\{m,V\}})$ is periodic. They establish a reduction rule in recurrence form by using Gessel-Xin's  product rule. On this basis, they give an explicit description of the sequence of Hankel determinants for any set $V$ of even elements of an even modulo $m$,
and they present a sufficient condition for the set $(m, V )$ such that the sequence of Hankel determinants is periodic.
However, there are still many instances with periodicity that are not covered by this sufficient condition.
They conjectured that if $m\geq 3$ and  $V=[m-1]$, then the
sequence of Hankel determinants is periodic, as restated in Corollary \ref{detm}.

We prove the following Theorem \ref{detr}, which includes their conjecture as a special case. We need some notations.
A \emph{periodic sequence} is written in contracted form using the notation with a star sign. Sometimes it is convenient to describe the structure of periodicity by using the form of an eventual periodic sequence. For instance, the sequence $(1,(0,-1,1)^*)$ represents $(1$, $0,-1,1$, $0,-1,1$, $0,-1,1$,$\dots)=(1,0,-1)^*$.
Denote by $H_{n\geq 1}^k(F):=(H_1^k(F(x)),H_2^k(F(x)),\dots)$ the \emph{$k$-shifted Hankel sequence} of $F(x)$. Then the Catalan generating function has nice Hankel sequences: $H_{n\ge 1}(c(x))=H_{n\ge 1}^1(c(x))=(1)^*$.

It is convenient to define $ F^{m,r}:=D^{(m,V)}=\sum_{n\ge 0}f_n^{m,r} x^n ,$  where $V=[m]\setminus \{r\}$.

\begin{theo}\label{detr}
 For any integer $m\geq 2$, the sequence of Hankel determinants of the series $F^{m,r}$ is periodic of the form:\\
when $r=1$,
\begin{equation}\label{det1}
  H_{n\geq 1}(F^{m,1})=
\left\{
    \begin{array}{lc}
        (1,\underbrace{0,\dots,0}_{m-1} ,1)^* &      {m \equiv 0,1 \quad(mod ~4)}\\
        (1,\underbrace{0,\dots,0}_{m-1} ,-1,-1,\underbrace{0,\dots,0}_{m-1},1)^* \quad  &    {m \equiv 2,3\quad (mod ~4)} ,\\
    \end{array}
\right.
\end{equation}
and when $r\geq 2$, by using the short notation $a \equiv_4 b$ for $a \equiv b~(mod~4)$, we have:
\begin{equation}\label{detrr}
\begin{footnotesize}
H_{n\geq 1}(F^{m,r})=
\left\{
    \begin{array}{lc}
        (1,\underbrace{0,\dots,0}_{r-2} ,1,\underbrace{0,\dots,0}_{m-r},1)^* &      {r \equiv_4 1,2 ~ and ~(m-r) \equiv_4 0,3 }\\
          (1,\underbrace{0,\dots,0}_{r-2} ,-1,\underbrace{0,\dots,0}_{m-r},1)^* &    {r \equiv_4 0,3 ~ and ~(m-r) \equiv_4 1,2 }\\
          (1,\underbrace{0,\dots,0}_{r-2} ,1,\underbrace{0,\dots,0}_{m-r},-1,-1,\underbrace{0,\dots,0}_{r-2},-1,\underbrace{0,\dots,0}_{m-r},1)^* \quad  &     {r \equiv_4 1,2 ~ and ~(m-r) \equiv_4 1,2 } \\
        (1,\underbrace{0,\dots,0}_{r-2} ,-1,\underbrace{0,\dots,0}_{m-r},-1,-1,\underbrace{0,\dots,0}_{r-2},1,\underbrace{0,\dots,0}_{m-r},1)^* \quad  &    {r \equiv_4 0,3 ~ and ~(m-r) \equiv_4 0,3 }. \\
    \end{array}
\right.
\end{footnotesize}
 \end{equation}

\end{theo}
In particular, when $r=m$, we have
\begin{cor}[{\citep[Conjecture 7.1]{CPF}}]\label{detm}
 For any integer $m\geq 2$, the sequence of Hankel determinants of the series $F^{m,m}$ is periodic of the form :
$$H_{n\geq 1}(F^{m,m})=
\left\{
    \begin{array}{lc}
        (1,\underbrace{0,\dots,0}_{m-2} ,1,1)^* &      {m \equiv 1,2 \quad(mod ~4)}\\
        (1,\underbrace{0,\dots,0}_{m-2} ,-1,-1,-1,\underbrace{0,\dots,0}_{m-2},1,1)^* \quad  &    {m \equiv 0,3\quad (mod ~4)} \\
    \end{array}
\right.
$$
\end{cor}

Note that the period of $ H_{\ge 1}(F^{m,r})$ is either $(m+1)$ or $2(m+1)$ for any $m\geq2 $.

We will prove their conjecture by using Sulanke-Xin's method.
The method was systematically used by Sulanke-Xin \cite{SX} for evaluating Hankel determinants of quadratic generating functions, such as
known results for
Catalan numbers, Motzkin numbers, Shr\"oder numbers, etc.
By using Gessel-Xin's constant rule and product rule, they derived a quadratic transformation $\tau$ such that
$H(F(x))$ and $H(\tau(F(x)))$ have simple connections. See section \ref{SXCF}.


The paper is organized as follows. Section 2 includes all the preparation work.
In Section 2.1, we introduce Gessel-Xin's continued fraction method and use it to obtain a special transformation for Hankel determinants of
$\dfrac{1}{1-ax-bxG(x)}$.  Then we obtain a slight extension of the $S$-fraction result, and
 enrich the class of eventually periodic Hankel determinant sequences when apply our transformation to $\D^S$.
Section 2.2 reviews the Sulanke-Xin's continued fraction method.
We deduce the functional equation for $F^{m,r}(x)$ in Section 2.3, and give a combinatorial proof of the functional equation of $F^{m,m}(x)$ in
Section 2.4. Finally in
Section 3, we prove Theorem \ref{detr} by evaluating the Hankel determinants $H_n(F^{m,r})$ by using Sulanke-Xin's continued fraction method.

\section{Preliminary}
\subsection{Gessel–Xin’s continued fraction method }\label{GXCF}
For an arbitrary two variable generating function $D(x,y)=\sum_{i,j=0}^{\infty}d_{i,j}x^iy^j$, let $[D(x,y)]_n$ be the determinant of the $\ n\times n$ matrix
$(d_{i,j})_{0\leq{i,j}\leq{n-1}}.$
There are three simple rules to transform the  sequence of determinants $\left([D(x, y)]_n\right)_{n\ge 0}$ to another sequence of determinants.

\noindent
\emph{Constant Rules.} Let $c$ be a non-zero constant. Then
$$[cD(x, y)]_n=c^n[D(x, y)]_n,\quad  \text{ and} \quad [D(cx, y)]_n = c^{\binom{n}{2}}[D(x, y)]_n=[D(x, cy)]_n.$$
\emph{Product Rules.} If $u(x)$ is any formal power series with $u(0) = 1$, then
$$[u(x)D(x, y)]_n = [D(x, y)]_n=[u(y)D(x, y)]_n.$$
\emph{Composition Rules.} If $v(x)$ is any formal power series with $v(0) = 0$ and $v'(0) = 1$, then
$$[D(v(x), y)]_n = [D(x, y)]_n=[D(x, v(y))]_n.$$

The constant rules are clear. The product and composition rules hold because the transformed determinants are obtained
from the original one by a sequence of elementary row or column operations. The composition rules are hard to use. Only several
examples are given in \cite{GX}.

Gessel-Xin's continued fraction method basically starts with the following observation on ordinary and shifted Hankel determinants:
\[H_n(A(x))=\left[\frac{xA(x)-yA(y)}{x-y}\right]_n,\qquad H^1_n(A(x))=\left[\frac{A(x)-A(y)}{x-y}\right]_n.\]
Then by applying the constant rules and product rules we will be able to obtain a recursion for evaluating $H(A(x))$ and $H^1(A(x))$.

\begin{lem}\label{ablem}
Suppose two generating functions $F(x)$ and $G(x)$ are related by $$F(x)=\frac{1}{1-ax-bxG(x)}, \qquad a,b\in \mathbb{C}, \ b\neq0.$$
Then $H_n(F(x))=(b)^{n-1}H_n^{1}(G(x))$ and $H_n^1(F(x))=H_n(a+bG(x))$.
\end{lem}
\begin{proof}
For ordinary Hankel determinants, we have
\begin{align*}
  H_n(F(x))=\left[\frac{\frac{x}{1-ax-bxG(x)}-\frac{y}{1-ay-byG(y)}}{x-y} \right]_n
\end{align*}
By the product rule, we can multiply by the series $(1-ax-bxG(x))(1-ay-byG(y))$ without changing the determinant, we obtain
\begin{align*}
  H_n(F(x)) & =\left[\frac{{x}(1-ay-byG(y))-{y}(1-ax-bxG(x))}{x-y} \right]_n \\
   & =\left[1+bxy\frac{G(x)-G(y)}{x-y} \right]_n=b^{n-1}H_n^{1}(G(x)),
\end{align*}
where in the last equality, we used the fact that $[1+xyD(x,y)]_n$ is a block diagonal determinant.

For shifted Hankel determinants, we have
$$H_n^1(F(x))=\left[\frac{\frac{1}{1-ax-bxG(x)}-\frac{1}{1-ay-byG(y)}}{x-y} \right]_n.$$
Similarly, we can multiply by the series $(1-ax-bxG(x))(1-ay-byG(y))$  to abtain
\begin{align*}
  H_n^1(F(x)) &= \left[\frac{(1-ay-byG(y))-(1-ax-bxG(x))}{x-y} \right]_n \\
    &=\left[\frac{x(a+bG(x))-y(a+bG(y))}{x-y} \right]_n=H_n(a+bG(x)).
\end{align*}
\end{proof}

By iteratively apply Lemma \ref{ablem}, we can extend the classical $S$-fraction result as follows.
\begin{prop}
Suppose we have the following continued fraction:
$$ F(x)= \dfrac{1}{1-a_1 x -\dfrac{b_1x}{1-a_2x-\dfrac{b_2x}{1-a_3x-\dfrac{b_3x}{\cdots}}}}.$$

\item [i)] If $a_i=0$ for all even indices $i$, then
$H_n(F(x))= (b_1 b_2)^{n-1} (b_3 b_4)^{n-2}\cdots (b_{2n-3} b_{2n-2}).$

\item [ii)] If $a_i=0$ for all odd indices $i$, then
$H_n^1(F(x))= b_1^{n} (b_2 b_3)^{n-1} \cdots (b_{2n-2} b_{2n-1}).$

In particular, if $a_i=0$ for all $i$, then we have the $S$-fraction result \text{\cite{W. B. Jones and W. J. Thron}}.
\end{prop}
Note that both parts of the proposition can be explained by the Gessel–Viennot–Lindström nonintersecting lattice model \cite{GV}.

Another consequence is the following.
\begin{cor}\label{cor-5}
For any $S\in \P$,
 we have
\begin{equation*}
  H_n(D^{{1}\bigcup (S+2)})=H_n(D^{S+2})=H_{n-1}(D^{S}),
  \quad and \quad
  H_n^1(D^{S+1})=H_{n-1}(D^{S}).
\end{equation*}

\end{cor}
\begin{proof}
By Equation (\ref{D22D}), if $S\in \P$ then we have 
\begin{equation}
D^{\{1\} \bigcup (S+2)}=
           \dfrac{1}{1+x-xD^{S+1}}, \quad
D^{S+2}=
           \dfrac{1}{1-xD^{S+1}}, \quad and \quad
D^{S+1}=
           \dfrac{1}{1-xD^{S}}.
\end{equation}
By Lemma 3, we have
$$H_{n}(D^{\{1\} \bigcup (S+2)})=H_{n}(D^{S+2})=H_{n-1}^1(D^{S+1}),\quad H_n^1(D^{S+1})=H_{n}(D^{S}).$$
The Corollary then follows.
\end{proof}

Repeated application of Corollary \ref{cor-5} gives the following result, enriching the class of eventually periodic
Hankel determinant sequences.
\begin{theo}
For any $S\in \P$, positive integer $p\in \P$, and  $T\subset [2p] \cap (2\Z+1)$. If $H_n(D^S(x))$ is eventual periodic, then

\item[ i)]
the Hankel sequence  $H_{n\geq1}(D^{T\bigcup (S+2p)}(x))$ is eventually periodic.

\item[ ii)] the shifted Hankel sequence  $H_{n\geq1}^1(D^{1+(T\bigcup (S+2p))}(x))$ is eventually periodic.

\end{theo}

\begin{example}

\item[i)] For $S=(5,\{1,2,4\})$,
we have $$\{1\}~\bigcup~ \left( S+2 \right) = (5,\{1,3,4\})~ \text{and}~ 1+(\{1\}~\bigcup~ \left( S+2 \right)) = (5,\{2,4,5\}).$$  Then,
$$H_n(D^{ \{ 5,\{1,2,4 \} \}})= H_{n+1}(D^{ \{ 5,\{1,3,4 \} \}})~\text{and}~
H_{n}(D^{ \{ 5,\{1,3,4 \} \}})=H_{n}^1(D^{ \{ 5,\{2,4,5 \} \}}).$$
This agrees with direction computation or Table 1 in \cite{CPF} that
\begin{align*}
H_{n\geq1}(D^{ \{ 5,\{1,2,4 \} \}})=(&1,0,-1,-1,-1,-1,0,1,1,1.)^*,\\
  H_{n\geq1}(D^{ \{ 5,\{1,3,4 \} \}})=(1,&1,0,-1,-1,-1,-1,0,1,1)^*,\\
  H_{n\ge1}^1(D^{ \{ 5,\{2,4,5 \} \}}=(1,&1,0,-1,-1,-1,-1,0,1,1)^*.
\end{align*}

\end{example}

\subsection{Sulanke-Xin's continued fraction}\label{SXCF}

We include Sulanke-Xin's continued fraction method here as our basic tool.
Suppose the generating function $F(x)$ is the unique solution of a quadratic functional equation which can be written as
\begin{gather}
  F(x)=\frac{x^d}{u(x)+x^kv(x)F(x)},\label{xinF(x)}
\end{gather}
where $u(x)$ and $v(x)$ are rational power series with nonzero constants, $d$ is a nonnegative integer, and $k$ is a positive integer.
We need the unique decomposition of $u(x)$ with respect to $d$: $u(x)=u_L(x)+x^{d+2}u_H(x)$ where $u_L(x)$ is a polynomial of degree at most $d+1$ and $u_H(x)$ is a power series.
Then Propositions 4.1 and 4.2 of \cite{SX} can be summarized as follows.
\begin{prop}\label{SXprop}
Let $F(x)$ be determined by  \eqref{xinF(x)}. Then the quadratic transformation $\tau(F)$ of $F$ defined as follows gives close connections
between   $H(F)$ and $H(\tau(F))$.
\begin{enumerate}
\item[i)] If $u(0)\neq1$, then $\tau(F)=G=u(0)F$ is determined by $G(x)=\frac{x^d}{u(0)^{-1}u(x)+x^ku(0)^{-2}v(x)G(x)}$, and $H_n(\tau(F))=u(0)^{n}H_n(F(x))$;

\item[ii)] If $u(0)=1$ and $k=1$, then $\tau(F)=x^{-1}(G(x)-G(0))$, where $G(x)$ is determined by
$$G(x)=\frac{-v(x)-xu_L(x)u_H(x)}{u_L(x)-x^{d+2}u_H(x)-x^{d+1}G(x)},$$
and we have
$$H_{n-d-1}(\tau(F))=(-1)^{\binom{d+1}{2}}H_n(F(x));$$

\item[iii)] If $u(0)=1$ and $k\geq2$, then $\tau(F)=G$, where $G(x)$ is determined by
$$G(x)=\frac{-x^{k-2}v(x)-u_L(x)u_H(x)}{u_L(x)-x^{d+2}u_H(x)-x^{d+2}G(x)},$$
and we have
$$H_{n-d-1}(\tau(F))=(-1)^{\binom{d+1}{2}}H_n(F(x)).$$\label{xinu(0)(ii)}

\end{enumerate}

\end{prop}

\subsection{The Functional Equation for $F^{m,r}(x)$ for $2\leq m\leq m,~1\leq r$}

%
To give the functional equation for $F^{m,r}(x)$, we
use the notation: $\sum_{i=0}^{r} x^i:= \frac{1-x^{r+1}}{1-x}$ for $r\in \Z$.
For Example, $\sum_{i=0}^{2} x^i =1+x+x^2$, $\sum_{i=0}^{-1} x^i =0$, $\sum_{i=0}^{-3} x^i =-x^{-1}-x^{-2}$.

\begin{theo}\label{mrcf}
 Let $m\geq 2$  and $ 1\leq r\leq m$, then
  \begin{footnotesize}
      \begin{equation*}
       F^{m,r}(x)=
       \frac{1-(\sum_{i=2}^{m-r+1}x^i )(\sum_{i=0}^{r-3}x^i)}
       {1+x-( \sum_{i=3}^{m-r+1} x^i)( \sum_{i=0}^{r-3} x^i) - (\sum_{i=2}^{m-r+1} x^i) ( \sum_{i=0}^{r-2} x^i ) +( \sum_{i=3}^{m-r+1} x^i \sum_{i=0}^{r-3} x^i -1   )xF^{m,r}(x)}.
      \end{equation*}
   \end{footnotesize}
 Especially, when $r=m$ we have
  \begin{equation}\label{mmcf}
    F^{m,m}(x)=\frac{1}{ 1+\sum_{i=1}^{m-1}x^i-( \sum_{i=1}^{m}x^i )F^{m,m}(x)};
  \end{equation}
when $r=1$, we have
  $$F^{m,1}(x)=\frac{\sum_{i=0}^{m-1}x^i}{{\sum_{i=0}^{m-1}x^i}-xF^{m,1}(x)}.$$
\end{theo}

 We will give a combinatorial proof of equation (\ref{mmcf}) in section \ref{CombinatorialProof}.
\begin{proof}
 Firstly, let $ \varphi_{n+1}(G(x))$ be the  $(n+1)$ layers continued fraction:
 \begin{equation*}
   \varphi_{n+1}(G(x))= \dfrac{1}{1+x-\dfrac{x}{1+x-\dfrac{x}{\dfrac{\ddots}{1+x-\dfrac{x}{1-xG(x)}}}}}.
 \end{equation*}
So, $\varphi_{n+1}(G(x))=\dfrac{1}{1+x-x\varphi_{n}(G(x))}$.
 It is easy to prove by induction that
 \begin{equation}\label{allx}
   \varphi_{n+1}(G(x))=  \frac{(\sum_{i=1}^{n-1}x^i )G(x)-1} {(\sum_{i=1}^{n}x^i)G(x) -1}.
 \end{equation}
 From  (\ref{basiccf}), we have $F^{m,r}(x)=\varphi_{r+1}(\varphi_{m-r+1}(F^{m,r}(x)-1))$. Using (\ref{allx}), we obtain:
\begin{equation}\label{crudf}
  F^{(m,r)}(x)- \frac{(\sum_{i=1}^{r-1}x^i )\left(\frac{(\sum_{i=1}^{m-r-1}x^i )(F^{m,r}(x)-1)-1} {(\sum_{i=1}^{m-r}x^i)(F^{m,r}(x)-1) -1} \right)-1} {(\sum_{i=1}^{r}x^i)\left(\frac{(\sum_{i=1}^{m-r-1}x^i )(F^{m,r}(x)-1)-1} {(\sum_{i=1}^{m-r}x^i)(F^{m,r}(x)-1) -1}\right) -1}=0.
\end{equation}
  Simplifying (\ref{crudf})  and then take the numerator, we obtain:

\begin{footnotesize}
 \begin{equation*}
  F^{m,r}(x)=
  \frac{1-(\sum_{i=2}^{m-r+1}x^i )(\sum_{i=0}^{r-3}x^i)}
  {1+x-( \sum_{i=3}^{m-r+1} x^i)( \sum_{i=0}^{r-3} x^i) - (\sum_{i=2}^{m-r+1} x^i) ( \sum_{i=0}^{r-2} x^i ) +( \sum_{i=3}^{m-r+1} x^i \sum_{i=0}^{r-3} x^i -1   )xF^{m,r}(x)}.
 \end{equation*}
\end{footnotesize}

\end{proof}
\subsection{ A Combinatorial Proof of Equation \eqref{mmcf}}\label{CombinatorialProof}
A Dyck path $\pi$ is said to be \emph{$m$-peaks} if the height of each peak is a multiple of $m$.
Equivalently, $\pi$ is $m$-peaks if and only if it avoids all peak hights in $(m,[m-1])$, i.e., $\pi \in \D^{(m,[m-1])}$.

We need the following bijection.
\begin{lem}\label{bbij}
For $1\leq k \leq m-1$ and $m\le n$.
  the number of $(n-k)$-Dyck paths $M$  such that $M$ is $m$-peaks is equal to
   the number of $n$-Dyck paths $N$ satisfying the following conditions:
   \begin{enumerate}[(i)]
     \item $N$ is $m$-peaks.
     \item $N$ does not return to the $x$-axis before the last step.
     \item The height of the rightmost valley under level $m$ is exactly $m-k$.
   \end{enumerate}
\end{lem}

\begin{proof}
Let $1\leq k \leq m-1$, $M \in \D_{n-k}$ and $M$ is $m$-peaks. Decompose $M$ as:
\[ M=M'M_1\]
where $M'$ first return to the $x$-axis.
So $M'$ is also $m$-peaks and must end with $m$ $\DD$s. Therefore, write $M'$ as $M_2\underbrace{\DD \DD\dots \DD}_m$. Now, $M$ is divided into three parts
$$M=M_2\underbrace{\DD\DD\dots \DD}_m M_1.$$
Define $N$ to be:
$$N=M_2\underbrace{\DD\DD\dots \DD}_k \underbrace{\U\U\dots \U}_k  M_1\underbrace{\DD\DD\dots \DD}_m.$$
It easy to check that $N$ is an $n$-Dyck paths  satisfying conditions $(i),(ii)$ and $(iii)$.

The inverse of $M\rightarrow N$: Find the right most valley as in condition$ (iii)$,
and  decompose $N$ as
$$ N=N_1 \underbrace{\DD\DD\dots \DD}_k ~ \underbrace{\U\U\dots \U}_k  N', $$
where the height of  terminal step of $N_1$  is $m$.
Clearly, $N'$ has no valley under level $m$ and ends with  $m \DD$s,
 Therefore we can decomposition $N'$ into $N_2 \underbrace{\DD\DD\dots \DD}_m,$ where $N_2$ is $m$-peaks.
 Now, $N$ is decomposed as
        $$ N=N_1 \underbrace{\DD\DD\dots \DD}_k ~ \underbrace{\U\U\dots \U}_k  N_2 \underbrace{\DD\DD\dots \DD}_m. $$
Then we set the corresponding $(n-k)$-Dyck path $M$ to be
$$ M=N_1  \underbrace{\DD\DD\dots \DD}_m  N_2. $$
Finally, for the above transformation, we always need $m\le n$.
\end{proof}
An example is shown in Figure \ref{fig}, where $m=3$ and $k=1$.
\begin{figure}[!ht]
  $$
  \hskip .1 in \vcenter{ \includegraphics[height=1.4 in]{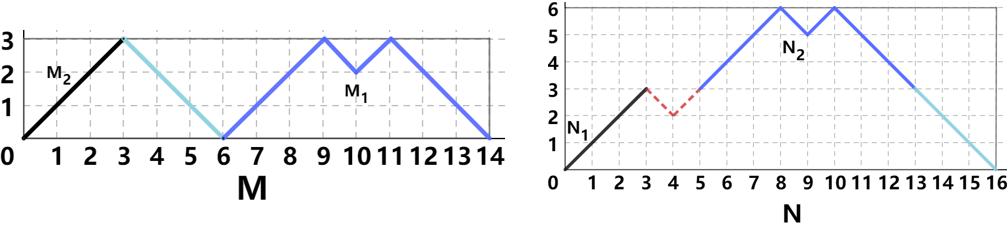}}
$$
\caption{$M$ and $N$ are 3-peaks. $M$ is a Dyck path with 7 up steps, $N$ is a Dyck path with 8 up steps.
\label{fig}}
\end{figure}

Observe that the height of the lowest valley of $N$ is $h\in [m-1]$
is equivalent to saying that the height of the rightmost valley of $N$ under level $m$ is equal to $m-k$
for some $k\in [m-1]$. Thus by taking all $1\leq k \leq m-1$ in Lemma \ref{bbij}, we will get the following result.

\begin{cor}\label{bi1}
For all positive integer $m$ with $2\leq m\leq n$, we have
$$\# \biguplus_{k=1}^{m-1}
\{\pi \in \D_{n-k}  \mid  \pi \textrm{ is  $m$-peaks} \}. $$
$$=\#\{\pi \in \D_n \mid
 \pi~ \text{is $m$-peaks and the height of the lowest valley of $\pi$ must be in } [m-1] \}.$$
\end{cor}

\begin{proof}[Combinatorial proof of Eq.\eqref{mmcf}]
For any nonempty $m$-peaks Dyck path $M$, the first return decomposition gives $$M=M_1 M_2.$$
Denote by  $h$  the height of lowest valley of $M_1$. Then $h>0$ since $M_1$ only return to the $x$-axis at the end.
Therefore we have the following two cases.

Case 1:  $m \leq h$. Clearly, such $M_1$ has generating function $x^mF^{m,m}(x)$;

Case 2: $1\leq h\le m-1$. By Corollary \ref{bi1} and Lemma \ref{bbij}, the generating function of such $M_1$ is  $\sum_{h=1}^{m-1}{x^h(F^{m,m}(x)-1)}$.

In summary, the generating function  $F^{m,m}(x)$ of $M$  satisfies the equation:
$$F^{m,m}(x)=1+F^{m,m}(x)\left(x^m F^{m,m}(x)+\sum_{i=1}^{m-1}{x^i(F^{m,m}(x)-1)}\right).$$
This can be rewrite as
 \[F^{m,m}(x)=\frac{1}{ 1+\sum_{i=1}^{m-1}x^i-( \sum_{i=1}^{m}x^i )F^{m,m}(x)} .\]

\end{proof}

\section{Proof of Theorem \ref{detr}}
We divide the proof of Theorem \ref{detr} into two parts: one is for $r=1$, and the other
 is for $r\ge 2$. This is simply due to the different periods in their formulas.

\subsection{Evaluation of $H_n(F^{m,1})$ }
  For $r=1$, we have the functional equation
  $$F^{m,1}(x)=\frac{\sum_{i=0}^{m-1}x^i}{{\sum_{i=0}^{m-1}x^i}-xF^{m,1}(x)}.$$
Let $F_0(x)=F^{m,1}(x)$. Apply Proposition \ref{SXprop} to obtain $F_1=\tau(F_0)$.
Firstly, $d=0,k=1,u=1$, Thus $u_L(x)=1$,  and $u_H=0$, $v=\frac{1}{-\sum_{i=0}^{m-1}x^i}$ then by $u(0)^{-1}=1$, we obtain:
$$H_n(F_0)=H_{n-1}(F_1),~ F_1(x)=\frac{x^{m-1}}{(\sum_{i=0}^{m-1}x^i)(1-2x-x^2F_1(x))}.$$
Apply Proposition \ref{SXprop} to get $F_2=\tau(F_1)$. This time $d=m-1$, $k=2$ and $u(x)$ is a polynomial:
$$u(x)=(\sum_{i=0}^{m-1}x^i)(1-2x),$$ and it then follows that $u_L(x)=u(x)$ and $u_H(x)=0 $. Then by $u(0)=1$, we obtain
$$H_{n-1}(F_1)=(-1)^{\binom{m}{2}}H_{n-m}(F_2) ,$$
$$ F_2(x)=\frac{\sum_{i=0}^{m-1}x^i}{1-\sum_{i=1}^{m-1}x^i-2x^m-x^{m+1} F_2(x)}.$$
Apply Proposition \ref{SXprop} to get $F_3=F_1=\tau(F_2)$. This time $d=0$, $k=m+1$, $u(x)=1-2x$ and it then follows that $u_L(x)=1-2x$ and $u_H(x)= 0$. Then by $u(0)=1$, we obtain
$$H_{n-m}(F_2)=H_{n-m-1}(F_3),$$
$$ F_3(x)=\frac{x^{m-1}}{(\sum_{i=0}^{m-1}x^i)(1-2x-x^2F_3(x))}.$$
By combining the above formulas we obtain
$$H_{n-1}(F_1)=(-1)^{\binom{m}{2}}H_{n-m-2}(F_1).$$
Let $n-1=k(m+1)+j$, where $0\leq j \leq m$. we deduce that
$$H_{n-1}(F_1)=(-1)^{k\binom{m}{2}}H_j(F_1).$$

The value of $\binom{m}{2}$ can be divided into two cases:
\begin{enumerate}[i)]
 \item When $m\equiv 0,1 $ (mod 4), $\binom{m}{2}$ is even.  In this case, $H_{k(m+1)+j}(F_1)=H_{j}(F_1)$ and the initial values are
       $$H_0(F_1)=1,H_1(F_1)=H_2(F_1)=\cdots=H_{m-1}(F_1)=0,H_{m}(F_1)=1.$$
  \item When $m\equiv 2 ,3 $ (mod 4), $\binom{m}{2}$ is odd.  In this case, $H_{k(m+1)+j}(F_1)=(-1)^k H_{j}(F_1)$  and the initial values are
       $$H_0(F_1)=1,H_1(F_1)=H_2(F_1)=\cdots=H_{m-1}(F_1)=0,H_{m}(F_1)=-1.$$
\end{enumerate}

This completes the proof of Eq. \eqref{det1} of Theorem \ref{detr}.

\subsection{Evaluation of $H_n(F^{m,r})$, $2\leq r \leq m$}
  Let $F_0(x)=F^{m,r}(x)$. Apply Proposition \ref{SXprop} to get $F_1=\tau(F_0)$. Firstly, $d=0,k=1$, and we need to decompose $u(x)$ with respect to $d$. We expand $u( x)$ as a power series and focus on
  the terms with small exponents ($\leq d+1 =1$)
 $$u(x)=\frac{1+x-\sum_{i=3}^{m-r+1} x^i \sum_{i=0}^{r-3} x^i -\sum_{i=2}^{m-r+1} x^i \sum_{i=0}^{r-2} x^i}{1-\sum_{i=2}^{m-r+1}x^i \sum_{i=0}^{r-3}x^i}.$$
Thus $u_L(x)=1+x$  is simple and $u_H=\frac{ u-u_L}{x^2}={\frac {{x}^{r-2}-{x}^{r-1}+{x}^{m-r+1}-{x}^{m-r}}{-{x}^{r}+{x}^{m}-{x
}^{m-r+2}+2\,x-1}}$. Then by $u(0)^{-1}=1$, we obtain
$$H_n(F_0(x))=F_{n-1}(F_1(x)),$$
$$F_1(x)=\frac{x^{r-2}}{x^2 F_1(x)(\sum_{i=2}^{m-r+1}x^i \sum_{i=0}^{r-3}x^i-1)+\sum_{i=r}^{m-1}x^i-\sum_{i=1}^{r-1}x^i+1}.$$
Apply Proposition \ref{SXprop} to get $F_2=\tau(F_1)$. This time $d=r-2$ and $u(x)$ is a polynomial:
$$u(x)= \sum_{i=r}^{m-1}x^i-\sum_{i=1}^{r-1}x^i+1,$$ and it then follows that $u_L(x)=-\sum_{i=1}^{r-1}x^i+1$ and $u_H(x)= \sum_{i=0}^{m-r-1}x^i $. Then by $u(0)=1$, we obtain
$$H_{n-1}(F_1)=(-1)^{\binom{r-1}{2}}H_{n-r}(F_2), $$
$$ F_2(x)=\frac{x^{m-r}}{1-\sum_{i=1}^{m-1}-x^r F_2(x)}.$$
Apply Proposition \ref{SXprop} to get $F_3=\tau(F_2)$. This time $d=m-r,~k=r\geq2$ and $u(x)$ is a polynomial: $$u(x)=1-\sum_{i=1}^{m-1}x^i,$$and it then follows that
$u_L=1-\sum_{i=1}^{m-r+1}x^i$ and $u_H=\sum_{i=0}^{r-3}x^i$. Then by $u(0)=1$, we obtain
$$H_{n-r}(F_2)=(-1)^{\binom{m-r+1}{2}}H_{n-m-1}(F_3),$$
$$F_3(x)=\frac{1-\sum_{i=2}^{m-r+1}x^i \sum_{i=0}^{r-3}x^i}{1-\sum_{i=1}^{m-r+1}x^i+\sum_{i=m-r+2}^{m-1}x^i-x^{m-r+2} F_3(x)}.$$
Apply Proposition \ref{SXprop} to get $F_4=F_1=\tau(F_3)$, This time $d=0,~k\geq2$ and $u(x)$, We need to decompose $u(x)$ with respect to $d$. We expand $u( x)$ as a power series and focus on
  the terms with small exponents ($\leq d+1 =1$):
$$u(x)=\frac{1-\sum_{i=1}^{m-r+1}x^i+\sum_{i=m-r+2}^{m-1}x^i}{1-\sum_{i=2}^{m-r+1}x^i \sum_{i=0}^{r-3}x^i}.$$
It then follows that $u_L=1-x$, $u_H={\frac {{x}^{1+r}-{x}^{r}-{x}^{m-r+3}+{x}^{m-r+2}}{{x}^{r}-{x}^{m}+{x}
^{m-r+2}-2\,x+1}}.$
Then by $u(0)=1$, we obtain :
$$H_{n-m-1}(F_3)=F_{n-m-2}(F_1),$$
By combining the above formulas we obtain
$$H_{n}(F^{m,r})=H_{n-1}(F_1)=(-1)^{\binom{r-1}{2}+\binom{m-r+1}{2}}H_{n-m-2}(F_1)$$
Let $n-1=k(m+1)+j$, where $0\leq j \leq m$. We deduce that
$$H_{n}(F^{m,r})=H_{n-1}(F_1)=(-1)^{k\binom{r-1}{2}+k\binom{m-r+1}{2}}H_j(F_1)$$
The value of $\binom{r-1}{2}+\binom{m-r+1}{2}$ can be divided into four cases:

\begin{enumerate}[i)]
  \item When $r\equiv_4 1,2$ and $(m-r) \equiv_4 0,3$, $\binom{r-1}{2}+\binom{m-r+1}{2}$ is even. In this case,
$H_{n-1}(F_1)=H_j(F_1)$ and the initial values are
  $$\left( H_n(F_1) \right)_{n=0}^{m}=(1,\underbrace{0,\dots,0}_{r-2} ,1,\underbrace{0,\dots,0}_{m-r},1).$$
  \item When $r\equiv_4 0,3$ and $(m-r) \equiv_4 1,2$, $\binom{r-1}{2}+\binom{m-r+1}{2}$ is even. In this case,
$H_{n-1}(F_1)=H_j(F_1)$ and the initial values are
  $$\left( H_n(F_1) \right)_{n=0}^{m}=(1,\underbrace{0,\dots,0}_{r-2} ,-1,\underbrace{0,\dots,0}_{m-r},1).$$
  \item When $r\equiv_4 1,2$ and $(m-r) \equiv_4 1,2$, $\binom{r-1}{2}+\binom{m-r+1}{2}$ is odd. In this case,
$H_{n-1}(F_1)=(-1)^kH_j(F_1)$ and the initial values are
  $$\left( H_n(F_1) \right)_{n=0}^{m}=(1,\underbrace{0,\dots,0}_{r-2} ,1,\underbrace{0,\dots,0}_{m-r},-1).$$
  \item When $r\equiv_4 0,3$ and $(m-r) \equiv_4 0,3$, $\binom{r-1}{2}+\binom{m-r+1}{2}$ is odd. In this case,
$H_{n-1}(F_1)=(-1)^kH_j(F_1)$ and the initial values are
  $$\left( H_n(F_1) \right)_{n=0}^{m}=(1,\underbrace{0,\dots,0}_{r-2} ,-1,\underbrace{0,\dots,0}_{m-r},-1).$$
\end{enumerate}

This completes the proof of Eq. \eqref{detrr} in Theorem \ref{detr}.


\begin{thebibliography}{99}

\bibitem{CPF} H.-L. Chien, S.-P. Eu and T.-S. Fu, On Hankel determinants for Dyck paths with peaks avoiding multiple classes of heights,	European Journal of Combinatorics, 101 (2022), 103478.

\bibitem{Cigler} J. Cigler, Some nice Hankel determinants, Arxiv preprint: 1109.1449, 2011.


\bibitem{ELY} S.-P. Eu, S.-C. Liu and Y.-N. Yeh, Dyck paths with peaks avoiding or restricted to a given set, Stud. Appl. Math. 111 (2003) 453--465.

\bibitem{GV} I. M. Gessel and G. Viennot, Binomial determinants, paths, and hook length formulae, Adv. Math. 58 (3) (1985) 300--321.

\bibitem{GX} I.M. Gessel and G. Xin, The generating function of ternary trees and continued fractions, Electron. J. Combin. 13 (2006) R53.

\bibitem{Han} G.-N. Han, Hankel continued fraction and its applications, Adv. Math. 303 (2016) 295--321.

\bibitem{W. B. Jones and W. J. Thron}  W. B. Jones and W. J. Thron, Continued Fractions: Analytic Theory and Applications, Encyclopedia of Mathematics and its Applications. vol. 11, Addison-Wesley Publishing Co., Reading, Mass., 1980.

\bibitem{C. Krattenthaler.}   C. Krattenthaler. Advanced determinant calculus: a complement. Linear Algebra Appl., 411 (2005) 68--166.

\bibitem{SX} R. A. Sulanke and G. Xin, Hankel determinants for some common lattice paths, Adv. Appl. Math. 40 (2008) 149--67.

\bibitem{EC2}  R. P. Stanley, Catalan numbers. Cambridge University Press, New York, 2015.
\bibitem{gauss} U. Tamm, Some aspects of Hankel matrices in coding theory and combinatorics, Electron. J. Combin. 8 (1) (2001) 31




\bibitem{H. S. Wall}  H. S. Wall. Analytic Theory of Continued Fractions. Van Nostrand, New York, 1948.






\end{thebibliography}
\end{document}